\documentclass[a4paper,12pt]{article}
\usepackage{graphicx,amssymb,amstext,amsmath}

\usepackage[latin1]{inputenc}
\usepackage{amsthm}
\usepackage{dsfont}
\usepackage{amsmath}
\usepackage{amsfonts}
\usepackage{amssymb}
\usepackage{mathrsfs}
\usepackage{graphicx}
\usepackage{graphicx,color}
\usepackage[all]{xy}
\usepackage[total={5.2in,9.1in},
top=1.4in, left=1.55in, includefoot]{geometry}

\newtheorem{prop}{Proposition}

\newcommand{\ds}{\displaystyle}
\newcommand{\B}{\mathbb{B}}
\newcommand{\A}{\mathbb{A}}
\newcommand{\re}{\mathbb{R}}
\newcommand{\und}{\underline}

\title{On the minima of Markov and Lagrange
Dynamical Spectra.}
\author{Carlos Gustavo T. de A. Moreira\footnote{\text{Partially supported by CNPq.}}\\ gugu@impa.br}
\begin{document}
\maketitle

\vskip .2 in
{\small \hskip 3.2in \it Dedicated to Jean-Christophe Yoccoz}
\vskip .2in
\begin{abstract}
\noindent We consider typical Lagrange and Markov dynamical spectra associated to horseshoes on surfaces. We show that for a large set of real functions on the surface, the minima of the corresponding Lagrange and Markov dynamical spectra coincide and are given by the image of a periodic point of the dynamics by the real function. This solves a question by Jean-Christophe Yoccoz.
\end{abstract}
 \section{Introduction}

\noindent The classical Lagrange spectrum (cf. \cite{CF}) is defined as  follows: Given an irrational number $\alpha$, according to Dirichlet's theorem the inequality $\left|\alpha-\frac{p}{q}\right|<\frac{1}{q^2}$ has infinite rational solutions $\frac{p}{q}$. 
Markov and Hurwitz improved this result (cf. \cite{CF}), proving that, for all irrational $\alpha$, the inequality $\left|\alpha-\frac{p}{q}\right|<\frac{1}{\sqrt{5}q^2}$ has infinitely many rational solutions $\frac{p}{q}$.\\

\noindent On the other hand, for a fixed  irrational $\alpha$, better results can be expected. We associate, to each $\alpha$, its best constant of approximation (Lagrange value of $\alpha$), given by \begin{eqnarray*}k(\alpha)&=&\sup\left\{k>0:\left|\alpha-\frac{p}{q}\right|<\frac{1}{kq^2} \ \text{has infinitely many rational solutions $\frac{p}{q}$ }\right\}\\
&=&\limsup_{ \stackrel{|p|,q\to\infty}{p\in \mathbb{Z},q\in \mathbb N}}\left|q(q\alpha-p)\right|^{-1}\in \re\cup\{+\infty\}.
\end{eqnarray*}
\noindent Then, we always have $k(\alpha)\geq \sqrt{5}$.
The \textit{{Lagrange spectrum}} is the set $$L=\{k(\alpha):\alpha\in \re\setminus \mathbb{Q} \ \text{and} \ k(\alpha)<\infty\}.$$

\noindent
Let $\alpha$ be an  irrational number expressed in continued fractions by $\alpha=[a_0,a_1,\dots]$. Define, for each $n\in \mathbb{N}$, $\alpha_{n}=[a_n,a_{n+1},\dots]$ and $\beta_{n}=[0,a_{n-1},a_{n-2},\dots]$. Using elementary continued fractions techniques it can be proved that 
$$k(\alpha)=\limsup_{n\to \infty}(\alpha_n+\beta_{n}).$$

The study of the geometric structure of $L$ is a classical subject, which began with
Markov, proving in 1879 (\cite{Ma}) that
$$
L\cap(-\infty, 3)=\large\{k_1=\sqrt 5< k_2=2 \sqrt 2<k_3= \frac{\sqrt{221}}5<\cdots\large\}
$$
where $k_n$ is a sequence (of irrational numbers whose squares are rational) converging to $3$ - more precisely, the elements $k_n$ of $L\cap(-\infty, 3)$ are the numbers the form $\sqrt{9-\frac4{z^2}}$, where $z$ is a positive integer such that there are other positive integers $x, y$ with $x\le y\le z$ and $x^2+y^2+z^2=3xyz$.

\noindent 	Another interesting set is the classical \textit{Markov spectrum} defined by (cf. \cite{CF})
\begin{equation*}\label{SMClassic}
M=\left\{\inf_{(x,y)\in \mathbb{Z}^{2}\setminus(0,0)}|f(x,y)|^{-1}:f(x,y)=ax^2+bxy+cy^2 \ \text{with} \ b^2-4ac=1\right\}.\end{equation*}

It is possible to prove (cf. \cite{CF}) that $L$ and $M$ are closed subsets of the real line with $L\subset M$ and that $L\cap(-\infty, 3)=M\cap(-\infty, 3)$.

\noindent Both the Lagrange and Markov spectrum have a dynamical interpretation.  This fact is an important motivation for our work.\\

\noindent
Let $\Sigma=({\mathbb{N}^*})^{\mathbb{Z}}$ and $\sigma\colon \Sigma \to \Sigma$ the shift defined by $\sigma((a_n)_{n\in\mathbb{Z}})=(a_{n+1})_{n\in\mathbb{Z}}$. If $f\colon \Sigma \to \re$ is defined by $f((a_n)_{n\in\mathbb{Z}})=\alpha_{0}+\beta_{0}=[a_0,a_1,\dots]+[0,a_{-1},a_{-2},\dots]$, then 

$$L=\left\{\limsup_{n\to\infty}f(\sigma^{n}(\underline{\theta})):\underline{\theta}\in \Sigma\right\}$$
and  
$$M=\left\{\sup_{n\in\mathbb{Z}}f(\sigma^{n}(\underline{\theta})):\underline{\theta}\in \Sigma\right\}.$$
\noindent Notice that $\sqrt{5}$, which is the common minimum of $L$ and $M$, is the image by $f$ of the fixed point $(\dots,1 ,1, 1,\dots)$ of the shift map $\sigma$. 

This last interpretation, in terms of a shift, admits a natural generalization of Lagrange and Markov spectrum in the context of hyperbolic dynamics (at least in dimension 2, which is the focus of this work).\\

\noindent
We will define, as in \cite{MR}, the Markov and Lagrange dynamical spectra associated to a hyperbolic set as follows. Let $M^2$ be a surface and $\varphi \colon M^2\to M^2$ be a diffeomorphism with $\Lambda\subset M^2$ a hyperbolic set for $\varphi$ (which means that $\varphi(\Lambda) = \Lambda$ and there is a decomposition $T_\Lambda M^2
= E^s\oplus E^u$ such that $D\varphi\vert_{E^s}$ is uniformly contracting and
$D\varphi\mid_{E^u}$ is uniformly expanding). In this paper we will asume that $\Lambda$ is a {\it horseshoe}: it is a compact, locally maximal, hyperbolic invariant set of saddle type (and thus $\Lambda$ is not an attractor nor a repellor, and is topologically a Cantor set). Let $f\colon M^2\to \re$ be 
a continuous real function, then the \textit{Lagrange Dynamical Spectrum} associated to $(f,\Lambda)$ is defined by 
$$\ds L(f,\Lambda)=\left\{\limsup_{n\to\infty}f(\varphi^{n}(x)):x\in \Lambda\right\},$$

\noindent
and the \textit{Markov Dynamical Spectrum} associate to $(f,\Lambda)$ is defined by  
$$M(f,\Lambda)=\left\{ \sup_{n\in \mathbb{Z}}f(\varphi^{n}(x)):x\in \Lambda\right\}.$$

Here we prove the following theorem, which solves a question posed by Jean-Christophe Yoccoz to the author in 1998:\\ 
\ \\
\noindent {\bf Main Theorem} \emph{Let $\Lambda$ be a horseshoe associated to a $C^2$-diffeomorphism $\varphi$. Then there is a dense set $H\subset C^{\infty}(M,\re)$, which is is $C^0$-open, such that for all $f\in H$, we have 
\begin{equation*}
\min  L(f,\Lambda)=\min M(f,\Lambda)=f(p),
\end{equation*}
where $p=p(f)\in \Lambda$ is a periodic point of $\varphi$. Moreover, $f(p)$ is an isolated point both in $L(f,\Lambda)$ and in $M(f,\Lambda)$}\\

\noindent $\bf{Remark}:$ In the previous statement, a horseshoe
means a compact, locally maximal (which means that it is the maximal invariant set in some neighbourhood of it), transitive hyperbolic invariant set of saddle type (and so it contains a dense subset of periodic orbits).

{\bf Acknowledgements:} I would like to thank Carlos Matheus, Davi Lima and Sandoel Vieira for helpful discussions on the subject of this paper. I would also like to thank the anonymous referee for his very valuable comments and suggestions, which helped to substantially improve this work.


\section {Preliminaries from dynamical systems}\label{Prel}

\noindent If $\Lambda$ is a hyperbolic set associated to a $C^{2}$-diffeomorphism, then the stable and unstable foliations $\mathcal{F}^{s}(\Lambda)$ and $\mathcal{F}^{u}(\Lambda)$ are $C^{1+\varepsilon}$ for some $\varepsilon>0$. Moreover, these foliations can be extended to $C^{1+\varepsilon}$ foliations defined on a full neighborhood of $\Lambda$(cf. the comments at the end of section 4.1 of \cite{PT}, pp. 60).\\
We will consider the following setting. Let $\Lambda$ be a horseshoe of $\varphi$. Let us fix a geometrical Markov partition $\{R_a\}_{a\in\mathbb{A}}$ with sufficiently small diameter consisting of rectangles $R_a\simeq I_a^s\times I_a^u$ delimited by compact pieces $I_a^s$, resp. $I_a^u$, of stable, resp. unstable, manifolds of certain points of $\Lambda$ (cf. \cite[pp 129]{Shub} or \cite{PT} for more details). The set $\mathbb{B}\subset \mathbb{A}^{2}$ of admissible transitions consist of pairs $(a_{0},a_{1})$ such that $\varphi(R_{a_{0}})\cap R_{a_{1}}\neq \emptyset$. So, we can define the following transition matrix $B$ which induces the same transitions than $\B\subset \A^{2}$
$$b_{a_ia_j}=1 \ \ \text{if} \ \  \varphi(R_{a_i})\cap R_{a_j}\neq \emptyset,\ \ \ b_{a_ia_j}=0 \ \ \ \text{otherwise, for $(a_i,a_j)\in \A^{2}$.}$$

\noindent Let $\Sigma_{\mathbb{A}}=\left\{\und{a}=(a_{n})_{n\in \mathbb{Z}}:a_{n}\in \mathbb{A} \ \text{for all} \ n\in \mathbb{Z}\right\}$. We can define the homeomorphism of $\Sigma_{\mathbb{A}}$, the shift, $\sigma:\Sigma_{\mathbb{A}}\to\Sigma_{\mathbb{A}}$ defined by $\sigma(({a}_{n})_{n\in \mathbb{Z}})=({a}_{n+1})_{n\in \mathbb{Z}}$. \\
Let $\Sigma_{B}=\left\{\und{a}\in \Sigma_{\mathbb{A}}:b_{a_{n}a_{n+1}}=1\right\}$, this set is a closed and $\sigma$-invariant subspace of $\Sigma_{\mathbb{A}}$. Still denote by $\sigma$ the restriction of $\sigma$ to $\Sigma_{B}$. The pair $(\Sigma_{B},\sigma)$ is called a subshift of finite type of $(\Sigma_\A, \sigma)$.
\noindent Given $x,y\in \A$, we denote by $N_{n}(x,y,B)$ the number of admissible strings for $B$ of length $n+1$, beginning at $x$ and ending with $y$. Then the following holds
$$N_{n}(x,y,B)=b^{n}_{xy}.$$  
In particular, since $\varphi|_{\Lambda}$ is transitive, there is  $N_{0}\in \mathbb{N}^*$ such that for all $x,y\in \A$, $N_{N_{0}}(x,y,B)>0$.

\noindent Subshifts of finite type also have a sort of local product structure. First we define the local stable and unstable sets: (cf. \cite[chap 10]{Shub})
\begin{eqnarray*}
W_{1/3}^{s}(\und{a})&=&\left\{\und{b}\in \Sigma_{B}: \forall n\geq 0, \ d(\sigma^{n}(\und{a}),\sigma^{n}(\und{b}))\leq 1/3\right\}\\
&=&\left\{\und{b}\in \Sigma_{B}: \forall n\geq 0, \ a_n=b_n\right\},\\
W_{1/3}^{u}(\und{a})&=&\left\{\und{b}\in \Sigma_{B}: \forall n\leq 0, \ d(\sigma^{n}(\und{a}),\sigma^{n}(\und{b}))\leq 1/3\right\}\\
&=&\left\{\und{b}\in \Sigma_{B}: \forall n\leq 0, \ a_n=b_n\right\},
\end{eqnarray*}
where $d(\und{a},\und{b})=\sum_{n=-\infty}^{\infty}2^{-(2\left|n\right|+1)}\delta_{n}(\und{a},\und{b})$ and $\delta_{n}(\und{a},\und{b})$ is $0$ when $a_n=b_n$ and $1$ otherwise.\\
So, if $\und{a},\und{b}\in \Sigma_{B}$ and $d(\und{a},\und{b})<1/2$, then $a_0=b_0$ and $W_{1/3}^{u}(\und{a})\cap W_{1/3}^{u}(\und{b})$ is a unique point, denoted by the bracket $[\und{a},\und{b}]=(\cdots,b_{-n},\cdots,b_{-1},b_{0},a_{1},\cdots,a_{n},\cdots)$.  

\noindent  If $\varphi$ is a diffeomorphism of a surface ($2$-manifold), then the dynamics of $\varphi$ on $\Lambda$ is topologically conjugate to a subshift $\Sigma_{B}$ defined by $B$, namely, there is a homeomorphism $\Pi\colon \Sigma_{B} \to \Lambda$ such that, the following diagram commutes
$$\xymatrix{\Sigma_{B} \ar[r]^\sigma \ar[d]_\Pi & \Sigma_{B} \ar[d]^{\Pi \ \ \ \ \ \ds \emph{i.e.}, \ \ \ \varphi\circ\Pi=\Pi\circ \sigma.}\\
\Lambda \ar[r]^\varphi & \Lambda }$$
Moreover, $\Pi$ is a morphism of the local product structure, that is, $\Pi[\und{a},\und{b}]=[\Pi(\und{a}),\Pi(\und{b})]$, (cf. \cite[chap 10]{Shub}).

If $p=\Pi(\theta)\in \Lambda$, we say that $\theta$ is the {\it kneading sequence} of $p$.

Next, we recall that the stable and unstable manifolds of $\Lambda$ can be extended to locally invariant $C^{1+\varepsilon}$-foliations in a neighborhood of $\Lambda$ for some $\varepsilon>0$. Therefore, we can use these foliations to define projections $\pi_a^u:R_a\to I_a^s\times\{i_a^u\}$ and $\pi_a^s: R_a\to \{i_a^s\}\times I_a^u$ of the rectangles into the connected components $I_a^s\times\{i_a^u\}$ and $\{i_a^s\}\times I_a^u$ of the stable and unstable boundaries of $R_a$ where $i_a^u\in\partial I_a^u$ and $i_a^s\in\partial I_a^s$ are fixed arbitrarily. Using these projections, we have the stable and unstable Cantor sets 
$$K^s=\bigcup\limits_{a\in\mathbb{A}}\pi_a^u(\Lambda\cap R_a) \quad \textrm{ and } \quad K^u=\bigcup\limits_{a\in\mathbb{A}}\pi_a^s(\Lambda\cap R_a)$$
associated to $\Lambda$. 

The stable and unstable Cantor sets $K^s$ and $K^u$ are $C^{1+\varepsilon}$-dynamically defined / $C^{1+\varepsilon}$-regular Cantor sets, i.e., the $C^{1+\varepsilon}$-maps 
$$g_s(\pi_{a_1}^u(y))=\pi_{a_0}^u(\varphi^{-1}(y))$$
for $y\in R_{a_1}\cap \varphi(R_{a_0})$
and
$$g_u(\pi_{a_0}^s(z))=\pi_{a_1}^s(\varphi(z))$$
for $z\in R_{a_0}\cap \varphi^{-1}(R_{a_1})$ are expanding of type $\Sigma_\mathbb{B}$ defining $K^s$ and $K^u$ in the sense that 
\begin{itemize}
\item the domains of $g_s$ and $g_u$ are disjoint unions $\bigsqcup\limits_{(a_0,a_1)\in\mathbb{B}} I^s(a_1, a_0)$ and $\bigsqcup\limits_{(a_0, a_1)\in\mathbb{B}} I^u(a_0, a_1)$ where $I^s(a_1, a_0)$, resp. $I^u(a_0, a_1)$, are compact subintervals of $I_{a_1}^s$, resp. $I_{a_0}^u$; 
\item for each $(a_0, a_1)\in\mathbb{B}$, the restrictions $g_s|_{I^s(a_0,a_1)}$ and $g_u|_{I^u(a_0,a_1)}$ are $C^{1+\varepsilon}$ diffeomorphisms onto $I^s_{a_0}$ and $I_{a_0}^u$ with $|Dg_s(t)|>1$, resp. $|Dg_u(t)|>1$, for all $t\in I^s(a_0, a_1)$, resp. $I^u(a_0, a_1)$ (for appropriate choices of the parametrization of $I_a^s$ and $I_a^u$); 
\item $K^s$, resp. $K^u$, are the maximal invariant sets associated to $g_s$, resp. $g_u$, that is, 
$$K^s=\bigcap\limits_{n\in\mathbb{N}}g_s^{-n}\left(\bigcup\limits_{(a_0,a_1)\in\mathbb{B}} I^s(a_1,a_0)\right) \textrm{ and } K^u=\bigcap\limits_{n\in\mathbb{N}}g_u^{-n}\left(\bigcup\limits_{(a_0,a_1)\in\mathbb{B}} I^u(a_0,a_1)\right)$$
\end{itemize} 
(see section 1 of chapter 4 of \cite{PT} and chapter 1 of \cite{MY1} for more informations on regular Cantor sets associated to horseshoes in surfaces).

Moreover, we will think the intervals $I^u_a$, resp. $I^s_a$, $a\in\mathbb{A}$ inside an abstract line so that it makes sense to say that the interval $I^u_a$, resp. $I^s_a$, is located to the left or to the right of the interval $I^u_b$, resp. $I^s_b$, for $a,b\in\mathbb{A}$. 

The stable and unstable Cantor sets $K^s$ and $K^u$ are closely related to the geometry of the horseshoe $\Lambda$: for instance, the horseshoe $\Lambda$ is locally diffeomorphic to the Cartesian product of the two regular Cantor sets $K^{s}$ and $K^{u}$ (since the stable and unstable foliations of $\Lambda$ are of class $C^1$). Moreover, the Hausdorff dimension of any regular Cantor set coincides with its box dimension. It follows that 
$$HD(\Lambda) = HD(K^s)+ HD(K^u)=:d_s+d_u$$
where $HD$ stands for the Hausdorff dimension (cf. Proposition 4 in \cite[chap. 4]{PT} and the comments before it).

In the paper \cite{CMM}, we were interested in the fractal geometry (Hausdorff dimension) of the sets $M(f, \Lambda)\cap (-\infty, t)$ and $L(f, \Lambda)\cap (-\infty, t)$ as $t\in\mathbb{R}$ varies. 

For this reason, we will also study the fractal geometry of 
$$\Lambda_t:=\bigcap\limits_{n\in\mathbb{Z}}\varphi^{-n}(\{y\in\Lambda: f(y)\leq t\}) = \{x\in\Lambda: m_{\varphi, f}(x)=\sup\limits_{n\in\mathbb{Z}}f(\varphi^n(x))\leq t\}$$ 
for $t\in\mathbb{R}$.

\section{Proofs}

We may study the subsets $\Lambda_t$ introduced above through its projections  
$$K_t^s=\bigcup_{a\in\mathbb{A}}\pi_a^u(\Lambda_t\cap R_a) \textrm{ and } K_t^u = \bigcup_{a\in\mathbb{A}}\pi_a^s(\Lambda_t\cap R_a)$$
on the stable and unstable Cantor sets of $\Lambda$.

It follows from the proof of Theorem 1.2 of \cite{CMM} (see remarks 1.3, 1.4 and 2.10 of \cite{CMM}) that (even in the non-conservative case) the box dimensions $D_s(t)$ of $K_t^s$ and $D_u(t)$ of $K_t^u$ depend continuously on $t$. In particular, if $t_0=\min M(f, \Lambda)$, then $D_s(t_0)=D_u(t_0)=0$, since, for any $t<t_0$, the sets $\Lambda_t$, $K_t^s$ and $K_t^u$ are empty, so $D_s(t)=D_u(t)=0$. Since the stable and unstable foliations of $\Lambda$ are of class $C^1$, the box dimension of $\Lambda_{t_0}$ is at most the sum of the box dimensions of $K_{t_0}^s$ and $K_{t_0}^u$, and so is equal to $0$.
It follows that, for any $\epsilon>0$, there is a locally maximal subhorseshoe (of finite type) $\tilde\Lambda\subset\Lambda$ with $\Lambda_{t_0}\subset\tilde\Lambda$ and $HD(\tilde\Lambda)<\epsilon$ (we may fix a large positive integer $m$ and take $\tilde\Lambda$ as the set of points of $\Lambda$ in whose kneading sequences all factors of size $m$ are factors of the kneading sequence of some element of $\Lambda_{t_0}$ - the number of such factors grow subexponentially in $m$ since the box dimension of $\Lambda_{t_0}$ is $0$, so the Hausdorff dimension of $\tilde\Lambda$ is small when $m$ is large by the estimates on fractal dimensions of \cite[chap. 4]{PT}).

\begin{prop}
Let $k>1$ be an integer. If $HD(\tilde\Lambda)<1/2k$, then, for any $r\in \mathbb{N}\cup \{\infty\}$,  there is a dense set (in the $C^r$ topology) of $C^r$ real functions $f$ such that, for some $c>0$, $|f(p)-f(q)|\ge c\cdot|p-q|^{k/(k-1)}, \forall p, q\in \tilde\Lambda$ (and in particular $f|_{\tilde\Lambda}$ is injective; moreover, its inverse function is $(1-1/k)$-H\"older).
\end{prop} 
\begin{proof}
Given a smooth function $f$, there are, as in Proposition 2.7 of \cite{CMM}, arbitrarily small perturbations of it whose derivative does not vanish at the stable and unstable directions in points of $\tilde\Lambda$, so we will assume that $f$ satisfies this property.

Given a Markov partition $\{R_a\}_{a\in\mathbb{A}}$ with sufficiently small diameter as before, we may perturb $f$ by adding, for each $a$, independently, a small constant $t_a$ to $f$ in a small neighbourhood of $\tilde\Lambda\cap R_a$ (notice that the compact sets $(\tilde \Lambda\cap R_a)$ are mutually disjoint). Since, for $a\ne b$, the images $f(\tilde\Lambda\cap R_a)$ and $f(\tilde\Lambda\cap R_b)$ have box dimensions smaller than $1/2k$, their arithmetic difference $f(\tilde\Lambda\cap R_a)-f(\tilde\Lambda\cap R_b)=\{x-y, x\in f(\tilde\Lambda\cap R_a), y\in f(\tilde\Lambda\cap R_b)\}=\{t\in\mathbb R|f(\tilde\Lambda\cap R_a)\cap(f(\tilde\Lambda\cap R_b)+t)\ne\emptyset$ has box dimension smaller than $1/k<1$, and so, for almost all $t_a, t_b$, the perturbed images $f(\tilde\Lambda\cap R_a)+t_a$ and $f(\tilde\Lambda\cap R_b)+t_b$ are disjoint.

Consider now parametrizations of small neighbourhoods of the pieces $R_a$ according to which the stable leaves of $\tilde\Lambda$ are $C^1$ close to be horizontal and the unstable leaves of $\tilde\Lambda$ are $C^1$ close to be vertical, and such that, in the coordinates given by these paramerizations, $f(x,y)$ is $C^1$ close to an affine map $f(x,y)=ax+by+c$, with $a$ and $b$ far from $0$ (these parametrizations exist since the pieces $R_a$ are chosen very small, so $f$ is close to be affine in $R_a$). Then we may consider, in each coordinate system as above, perturbations of $f$ of the type $f_{\lambda}(x,y)=f(x,\lambda y)$, where $\lambda$ is a parameter close to $1$. 

Given $\rho>0$ small, a $\rho$-decomposition of $\tilde\Lambda\cap R_a$ is a decomposition of it in a union of rectangles $I_j^s\times I_j^u$ intersected with $\tilde\Lambda$ such that both intervals $I_j^s$ and $I_j^u$ have length of the order of $\rho$. Let $r$ be a large positive integer, and consider $2^{-kr}$ and $2^{-(k-1)r}$-decompositions of $\tilde\Lambda\cap R_a$. Given two rectangles of the $2^{-kr}$-decomposition which belong to different rectangles of the $2^{-(k-1)r}$-decomposition (and so have distance at least of the order of $2^{-(k-1)r}$), the measure of the interval of values of $\lambda$ such that the images by $f_{\lambda}$ of the two rectangles of the $2^{-kr}$-decomposition have distance smaller than $2^{-kr}$ is at most of the order of $2^{-kr}/2^{-(k-1)r}=2^{-r}$. Since the box dimension of $\tilde\Lambda$ is $d<1/2k$, the number of pairs of rectangles in the $2^{-kr}$-decomposition is of the order of $(2^{-kr})^{-2d}=2^{2dkr}$, and so the the measure of the set of values of $\lambda$ such that the images of some pair as before of two rectangles of the $2^{-kr}$-decomposition have non-empty intersection is at most of the order of $2^{2dkr}\cdot 2^{-r}=2^{-(1-2dk)r}\ll 1$ (notice that $2dk<1$). The sum of these measures for all $r\ge r_0$ is $O(2^{-(1-2dk)r_0})\ll 1$, and so there is $\lambda$ close to $1$ such that, if $\epsilon$ is small and $p, q\in \tilde\Lambda\cap R_a$ are such that $|p-q|\ge \epsilon$ then $|f(p)-f(q)|$ is at least of the order of $\epsilon^{k/(k-1)}$ (consider $r$ in the above discussion such that $2^{-(k-1)r}$ is of the order of $\epsilon$). This implies the result: for some $c>0$, $|f(p)-f(q)|\ge c\cdot|p-q|^{k/(k-1)}, \forall p, q\in \tilde\Lambda$. It follows that $f|_{\tilde\Lambda}$ is injective and its inverse function $g$ is $(1-1/k)$-H\"older: indeed, it satisfies $|g(x)-g(y)|\le (c^{-1}|x-y|)^{(1-1/k)}$, for any $x, y\in f(\tilde\Lambda)$.
\end{proof}




Let $\varphi:M\rightarrow M$ be a diffeomorphism of a compact 2-manifold $M$ and let $\Lambda$ be a horseshoe for $\varphi$.

\ \\
We recall the following remark from \cite{MR}:

\noindent{\bf{Remark:}} We have  $L(f,\Lambda)\subset M(f,\Lambda)$ for any $f\in C^{0}(M,\mathbb{R})$. 
In fact:\\
Let $a\in L(f,\Lambda)$, then there is $x_0\in \Lambda$ such that $\ds a=\limsup_{n \to +\infty}f(\varphi^{n}(x_0))$. Since $\Lambda$ is a compact set, then there is a subsequence $(\varphi^{n_{k}}(x_0))$ of $(\varphi^{n}(x_0))$ such that $\ds\lim_{k\to +\infty}\varphi^{n_k}(x_0)=y_{0}$ and 
$$a=\limsup_{n \to +\infty}f(\varphi^{n}(x_0))=\lim_{k\to +\infty}f(\varphi^{n_k}(x_0))=f(y_0).$$

\noindent {\it{Claim:}} $f(y_{0})\geq f(\varphi^{n}(y_0))$ for all $n\in \mathbb{Z}$. Otherwise, suppose there is $n_{0}\in \mathbb{Z}$ such that $f(y_{0})<f(\varphi^{n_0}(y_{0}))$. Put $\epsilon=f(\varphi^{n_0}(y_{0}))-f(y_0)$, then, since $f$ is a continuous function, there is a neighborhood $U$ of $y_0$ such that 
$$f(y_{0})+\frac{\epsilon}{2}<f(\varphi^{n_0}(z)) \ \text{for all} \ z\in U.$$
Thus, since $\varphi^{n_k}(x_0)\to y_0$, then there is $k_{0}\in \mathbb{N}$ such that $\varphi^{n_k}(x_0)\in U$ for $k\geq k_0$, therefore, 
$$f(y_{0})+\frac{\epsilon}{2}<f(\varphi^{n_0+n_k}(x_0)) \ \text{for all} \ k\geq k_0.$$
This contradicts the definition of $a=f(y_0)$.\\
\ \\

\begin{prop}
Assume that $f|_{\tilde\Lambda}$ is injective. Then $\min  L(f,\Lambda)=\min M(f,\Lambda)=f(p)$ for only one value of $p\in \tilde\Lambda$ such that the restriction of $\varphi$ to the closure of the orbit of $p$ is minimal.
\end{prop}

\begin{proof}
Let $p\in \Lambda$ be such that $f(p)=\min M(f,\Lambda)$, which is unique since $f|_{\tilde\Lambda}$ is injective. We have $f(\varphi^j(p))<f(p)$ for all integer $j$ such that $\varphi^j(p)\ne p$. If some subsequence $\varphi^{n_k}(p)$ converges to a point $q$ such that $p$ does not belong to the closure of the orbit of $q$, $f(p)$ does not belong to the image by $f$ of the closure of the orbit of $q$ and so the Markov value of the orbit of $q$ is strictly smaller than $f(p)$, a contradiction. This implies that the restriction of $\varphi$ to the closure of the orbit of $p$ is minimal and, in particular, $f(p)$ is the Lagrange value of its orbit, so we also have $f(p)=\min  L(f,\Lambda)$. 
\end{proof}

Let $\Lambda$ be a horseshoe associated to a $C^2$-diffeomorphism $\varphi$. We define $X\subset C^0(M,\re)$ as the set of real functions $f$ for which
\begin{equation*}
\min  L(f,\Lambda)=\min M(f,\Lambda)=f(p),
\end{equation*}
where $p=p(f)\in \Lambda$ is a periodic point of $\varphi$ and there is $\varepsilon>0$ such that, for every $q\in\Lambda$ which does not belong to the orbit of $p$, $sup_{n\in \mathbb{Z}}f(\varphi^{n}(q))>f(p)+\varepsilon$.

\begin{prop} 
$X$ is open in $C^0(M,\re)$
\end{prop} 
\begin{proof}
Suppose that $f\in X$. If $\varepsilon>0$ is as in the definition of $X$, let $g\in C^0(M,\re)$ such that $|g(x)-f(x)|<\varepsilon/3, \forall x\in M$. Then we have $sup_{n\in \mathbb{Z}}g(\varphi^{n}(p))=g(\tilde p)<f(\tilde p)+\varepsilon/3\le f(p)+\varepsilon/3$, for some point $\tilde p$ in the (finite) orbit of $p$. Moreover, for every $q\in\Lambda$ which does not belong to the orbit of $p$, since $g(\varphi^{n}(q))>f(\varphi^{n}(q))-\varepsilon/3$, we have $sup_{n\in \mathbb{Z}}g(\varphi^{n}(q))\ge sup_{n\in \mathbb{Z}}f(\varphi^{n}(q))-\varepsilon/3>f(p)+\varepsilon-\varepsilon/3=f(p)+2\varepsilon/3>g(\tilde p)+\varepsilon/3$. So we have $\min  L(g,\Lambda)=\min M(g,\Lambda)=g(\tilde p)$ and $g\in X$.
\end{proof}

{\bf Proof of the Main Theorem:}
Let $t_0=\min M(f, \Lambda)$. Fix a large positive integer $K$ and consider a locally maximal subhorseshoe $\tilde\Lambda\subset\Lambda$ with $\Lambda_{t_0}\subset\tilde\Lambda$ and $HD(\tilde\Lambda)<1/2K$. We take symbolic representations of points of $\Lambda$ associated to a Markov partition of $\Lambda$. We will assume that $f$ satisfies the conclusions of Proposition 1 (replacing $k$ by $K$). We will prove that, under these conditions, we have $f\in X$, which will conclude the proof, since, if $f\in X$, then clearly $\min M(f, \Lambda)=\min L(f, \Lambda)$ is isolated in $M(f, \Lambda)$, and thus also in $L(f, \Lambda)$ (and $X$ is $C^0$-open, by Proposition 3). 

Since the restriction of $f$ to $\tilde\Lambda$ is injective, by Proposition 2 there is a unique $p\in \Lambda$ such that $f(p)=t_0$, and the restriction of $\varphi$ to the closure of the orbit of $p$ is minimal. Let $\theta=(\dots, a_{-2}, a_{-1}, a_0, a_1, a_2,\dots)$ be the kneading sequence of $p$. Assume by contradiction that $p$ is not a periodic point.

We will consider the following regular Cantor sets, which we may assume, using parametrizations, to be contained in the real line:
$K^s=W^s_{\text loc}(p)\cap \tilde\Lambda$, the set of points of $\tilde\Lambda$ whose kneading sequences are of the type $(\dots, b_{-2}, b_{-1}, a_0, a_1, a_2,\dots)$, for some $b_{-1}, b_{-2},\dots$  
(the point corresponding to this sequence will be denoted by $\pi^s(b_{-1},b_{-2},\dots)$)
 and $K^u=W^u_{\text loc}(p)\cap \tilde\Lambda$, the set of points of $\tilde\Lambda$ whose kneading sequences are of the type $(\dots, a_{-2}, a_{-1}, a_0, b_1, b_2,\dots)$, for some $b_1, b_2,\dots$ 
(the point corresponding to this sequence will be denoted by $\pi^u(b_1,b_2,\dots)$)
. Given a finite sequence $(c_{-1},c_{-2},\dots,c_{-r})$ such that $(c_{-r},\dots,c_{-2},c_{-1},a_0)$ is admissible, we define the interval $I^s(c_{-1},c_{-2},\dots,c_{-r})$ to be the convex hull of 
$$\{(\dots, b_{-2}, b_{-1}, a_0, a_1, a_2,\dots)\in K^s|b_{-j}=c_{-j},1\le j\le r\}.$$ 
Analogously, given a finite sequence $(d_1,d_2,\dots,d_s)$ such that $(a_0,d_1,d_2,\dots,d_s)$ is admissible, we define the interval $I^u(d_1,d_2,\dots,d_s)$ to be the convex hull of 
$$\{(\dots, a_{-2}, a_{-1}, a_0, b_1, b_2,\dots)\in K^u|b_j=d_j,1\le j\le s\}.$$ 

If $a$ and $b$ are the values of the derivative of $f$ at $p$ applied to the unit tangent vectors of $W^s_{\text loc}(p)$ and $W^u_{\text loc}(p)$, respectively, we have that $a$ and $b$ are non-zero and, considering local isometric parametrizations of $W^s_{\text loc}(p)$ and $W^u_{\text loc}(p)$ which send $p$ to $0$, we have that, locally, for $x\in W^s_{\text loc}(p)\supset K^s$ and $y\in W^u_{\text loc}(p)\supset K^u$, $f(x,y)=ax+by+O(x^2+y^2)$, in coordinates given by extended local stable and unstable foliations of the horseshoe $\tilde\Lambda$ (here, the point $p$ has coordinates $(x,y)=(0,0)$; we will use this local form in small neighbourhoods of $p$, i.e, for $|x|$ and $y$ small). 

Let $k\ne 0$ such that $a_k=a_0$. We define 
$$d_k=\min\{d(\pi^s(a_{k-1},a_{k-2},\dots),p), d(\pi^u(a_{k+1},a_{k+2},\dots),p)\}.$$ 
There are $0<\lambda_1<\lambda_2<1$ such that the norm of the derivative of $\varphi$ restricted to a stable direction and the inverse of the norm  of the derivative of $\varphi$ restricted to a unstable direction always belong to $(\lambda_1,\lambda_2)$.
We say that $k>0$ is a {\it weak record} if $d_k<d_j$ for all $j$ with $1\le j<k$ such that $a_j=a_0$. We will construct the sequence $0<k_1<k_2<\dots$ of the {\it records}. Let $k_1$ be the smallest $k>0$ with $a_k=a_0$ and, given a record $k_n$, $k_{n+1}$ will be the smallest weak record $k>k_n$ with $d_k<\min\{|a/b|,|b/a|\}\cdot \lambda_1^3 \cdot d_{k_n}$.
We say that $k>0$ such that $a_k=a_0$ is {\it left-good} if $f(\pi^s(a_{k-1},a_{k-2},\dots))<f(p)$ and that $k$ is {\it right-good} if $f(\pi^u(a_{k+1},a_{k+2},\dots))<f(p)$. We say that $k$ is {\it left-happy} if $k$ is left-good but not right-good or if $k$ is left-good and right-good and $d(\pi^s(a_{k-1},a_{k-2},\dots),p)\ge d(\pi^u(a_{k+1},a_{k+2},\dots),p)$. We say that $k$ is {\it right-happy} if $k$ is right-good but not left-good or if $k$ is left-good and right-good and $d(\pi^s(a_{k-1},a_{k-2},\dots),p)<d(\pi^u(a_{k+1},a_{k+2},\dots),p)$.  Notice that $k$ cannot be simultaneously left-happy and right-happy. We say that an index $k>0$ is {\it cool} if it is left-happy or right-happy and $a_{k+j}=a_j$ for every $j$ with $|j|\le K$. By the minimality, we have $\lim d_{k_n}=0$, and, since $f(\varphi^j(p))<f(p)$ for all $j\ne 0$, for all large values of $n$, $k_n$ is cool.

If $k$ is a cool index then we define its {\it basic cell} as follows: if $k$ is left-happy, we take $r_k$ to be the positive integer $r$ such that $a_{k-j}=a_{-j}$ for $0\le j<r$ and  $a_{k-r}\ne a_{-r}$, and $s_k$ to be the smallest positive integer $s$ such that $\min\{|a/b|,|b/a|\}\cdot \lambda_1^3|I^u(a_{k+1},a_{k+2},\dots,a_{k+s})|\le |I^s(a_{k-1},a_{k-2},\dots,a_{k-r_j})|$. Analogously, if $k$ is right-happy, we take $s_k$ to be the positive integer $s$ such that $a_{k+j}=a_{j}$ for $0\le j<s$ and   
$a_{k+s}\ne a_{s}$, and $r_k$ to be the smallest positive integer $r$ such that 
$\min\{|a/b|,|b/a|\}\cdot \lambda_1^3|I^s(a_{k-1},a_{k-2},\dots,a_{k-r})|\le |I^u(a_{k+1},a_{k+2},\dots,a_{k+s_k})|$. 
The basic cell of $k$ is the finite sequence \hfill\break $(a_{k-r_k},\dots,a_{k-1},a_k,a_{k+1},\dots,a_{k+s_k})$ indexed by the interval $[-r_k,s_k]$ of integers (so that the index $0$ in this interval corresponds to $a_k$). Notice that $r_{k_{n+1}}>r_{k_n}$ and $s_{k_{n+1}}>s_{k_n}$ for every $n$ large. We define the {\it extended cell} of $k$ as the finite sequence \hfill\break $(a_{k-\tilde r_k},\dots,a_{k-1},a_k,a_{k+1},\dots,a_{k+\tilde s_k})$ indexed by the interval $[-\tilde r_k,\tilde s_k]$ of integers, where $\tilde r_k:=\lfloor(1+2/K)r_k\rfloor$ and $\tilde s_k:=\lfloor(1+2/K)s_k\rfloor$. A crucial remark is that, since $f$ satisfies the conclusions of Proposition 1 and $f(\varphi^j(p))<f(p)$ for all $j\ne 0$, \hfill\break
$\bullet$ There is a positive integer $r_0$ such that for every $m>0$ which is not cool (in particular if $a_m\ne a_0$), and any point $q\in \tilde\Lambda$ whose kneading sequence $(\dots,b_{-1},b_0,b_1,\dots)$ satisfies $b_j=a_{m+j}$ for $-r_0\le j\le r_0$, we have $f(q)<f(p)$. Indeed, there is a constant $\tilde r$ such that, if $m>0$ is not cool, then $(a_{m-\tilde r},\dots,a_{m-1},a_m,a_{m+1},\dots,a_{m+\tilde r})\ne (a_{-\tilde r},\dots,a_{-1},a_0,a_{1},\dots,a_{\tilde r})$. We have $f(\varphi^m(p))<f(p)$, and, if $r_0$ is much larger than $\tilde r$, if $q$ is a point in $\tilde\Lambda$ whose kneading sequence $(\dots,b_{-1},b_0,b_1,\dots)$ satisfies $b_j=a_{m+j}$ for $-r_0\le j\le r_0$, we have $f(q)$ much closer to $f(\varphi^m(p))$ than to $f(p)$, and so $f(q)<f(p)$ (recall that $f$ is injective in $\tilde\Lambda$).  \hfill\break
$\bullet$ For any cool index $k>0$, if $(a_{k-\tilde r_k},\dots,a_{k-1},a_k,a_{k+1},\dots,a_{k+\tilde s_k})$ is the extended cell of $k$, and $q\in \tilde\Lambda$ is any point whose kneading sequence $(\dots,b_{-1},b_0,b_1,\dots)$ satisfies $b_j=a_{k+j}$ for $-\tilde r_k\le j\le \tilde s_k$, we have $f(q)<f(p)$. Indeed, since $|f(p)-f(q)|\ge c\cdot|p-q|^{k/(k-1)}, \forall p, q\in \tilde\Lambda$, by definition of the extended cells, $f(p)$ does not belong tho the convex hull of the image by $f$ of the image by $\Pi$ of the cylinder $\{(\dots,b_{-1},b_0,b_1,\dots)| b_j=a_{k+j}, -\tilde r_k\le j\le \tilde s_k\}$, which contains the point $f(\varphi^k(p))<f(p)$.

We now show the following
 
\noindent {\it{Claim:}} There is a positive integer $m$ such that we never have $(a_{-mt},\dots,a_{-1},a_0,a_1,\dots,a_{mt-1})=\gamma^{2m}=\gamma\gamma\dots\gamma$ ($2m$ times), for any finite sequence $\gamma$, where $t=|\gamma|$. 

Indeed, take a positive integer $k_0$ such that $\lambda_2^{k_0}<\lambda_1$, a large positive integer $m_0$ (with $m_0\ge\max\{r_0,K\}$) and $m=m_0 k_0$. Suppose by contradiction that \hfill\break $(a_{-mt},\dots,a_{-1},a_0,a_1,\dots,a_{mt-1})=\gamma^{2m}$ for some $\gamma=(c_1,c_2,\dots,c_t)$. We may assume that $\gamma$ is not of the form $\alpha^n$ for a smaller sequence $\alpha$ (otherwise we may replace $\gamma$ by $\alpha$). The Markov value $\tilde t_0$ of $\Pi(\overline\gamma)=\Pi(\dots\gamma\gamma\gamma\dots)$, which is larger than $t_0$, is attained at $c_1$. Indeed, if $2\le j\le t$, either $j-1$ is not cool or the extended cell of $j-1$ (centered in $c_j$) is contained in $\gamma^3$, so $f(\varphi^{j-1}(\Pi(\overline\gamma)))<f(p)=t_0$. Take the maximum values of $m_1, m_2\ge m$ for which $(a_{-m_1 t},\dots,a_{-1},a_0,a_1,\dots,a_{m_2 t-1})=\gamma^{m_1+m_2}$. Since $f(\varphi^{-(m_1-m_0)t}(p))$, $f(\varphi^{-(m_1-m_0-1)t}(p))$, $f(\varphi^{(m_2-m_0-1)t}(p))$ and $f(\varphi^{(m_2-m_0)t}(p))$ are smaller than $f(p)=t_0$, it follows that, for every $j\ge 1$, $f(\pi^u(\gamma^j a_{m_2t}, a_{m_2 t+1},a_{m_2 t+2},\dots))<f(\pi^u(\gamma^{j+1} a_{m_2t}, a_{m_2 t+1},a_{m_2 t+2},\dots))$ (the suffix $(a_{m_2t}, a_{m_2 t+1},a_{m_2 t+2},\dots)$ helps diminishing the value of $f$) and \hfill\break $f(\pi^s((\gamma^t)^j a_{-m_1 t-1},a_{-m_1 t-2},\dots))<f(\pi^s((\gamma^t)^{j+1} a_{-m_1 t-1},a_{-m_1 t-2},\dots))$, where \hfill\break $\gamma^t=(c_t,\dots,c_2,c_1)$ (the prefix $(a_{-m_1 t-1},a_{-m_1 t-2},\dots)$ also helps diminishing the value of $f$). Thus, by comparison, and using the previous remark on extended cells, if we delete from each factor of $(\dots,a_{-1},a_0,a_1,\dots)$ equal to $(a_{-(m_1+1) t},\dots,a_{-1},a_0,a_1,\dots,a_{(m_2+1) t-1})$ the factor $(a_0,a_1,\dots,a_{t-1})=\gamma$, we will reduce the Markov value of the sequence, a contradiction (notice that the number of consecutive copies of $\gamma$ in a factor of $(\dots,a_{-1},a_0,a_1,\dots)$ is bounded, since $\tilde t_0>t_0$). This concludes the proof of the Claim.\\

Notice that we can assume that $K$ is much larger than $m^2$ ($K>5m^2$ is enough for our purposes), by reducing $\tilde\Lambda$, if necessary.

In order to conclude the proof we will have two cases: 

(i) There are arbitrarily large values of $n$ for which $k_n$ is left-happy.

In this case we will show that, for such a large value of $n$, the periodic point $q$ of $\tilde\Lambda$ whose kneading sequence $(\dots, b_{-1},b_0,b_1,\dots)$ has period $(a_0,a_1,\dots,a_{k_n-1})$ (and so satisfies $b_m=a_{m\pmod{k_n}}$, for all $m$; notice that, since $\tilde\Lambda$ is locally maximal and $p\in\tilde\Lambda$, we have $q\in \tilde\Lambda$ for $n$ large) has Markov value smaller than $t_0$, a contradiction.  In order to do this, notice that if $r_{k_n}>2m^2 k_n$, then we get a contradiction by the previous Claim. If there is $0<k<k_n$ which is cool and satisfies $r_k>2m^2 k$ or $s_k>2m^2(k_n-k)$, we also get a contradiction by the previous Claim. Otherwise, except perhaps for the terms whose indices are multiple of $k_n$, any term equal
to $a_0$ in the periodic sequence with period $(a_0, a_1,\dots,a_{k_n-1})$ has a neighbourhood in this
periodic sequence which coincides with the extended cell of a corresponding element of the
original sequence $(\dots, a_{-1}, a_0, a_1,\dots)$, and so the Markov value of the point corresponding
to this periodic sequence is smaller than $t_0$. For the terms whose indices are multiple of $k_n$ (which correspond to the point $q$), if $\hat s$ is the positive integer such that $a_{k_n+j} = a_j$ for $0 \le j < \hat s$ and $a_{k_n+\hat s}\ne a_{\hat s}$, then $b_j=a_j$ for all $0\le j<k_n+\hat s$, and so $d(\pi^u(b_1, b_2,\dots), p)=o(d(\pi^u(a_{k_n+1}, a_{k_n+2},\dots), p))=o(d(\pi^s(a_{k_n-1}, a_{k_n-2},\dots), p))$, which implies $f(q)<f(p)=t_0$ (since $f(x,y) = ax + by + O(x^2 + y^2))$.

(ii) For all $n$ large, $k_n$ is right-happy. 

In this case we will show that, for $n$ large, the periodic point $q'$ of $\tilde\Lambda$ whose kneading sequence $(\dots, b'_{-1},b'_0,b'_1,\dots)$ has period $(a_{k_n},a_{k_n+1},\dots,a_{k_{n+1}-1})$ (and so satisfies $b_m=a_{m\pmod{k_{n+1}-k_n}+k_n}$, for all $m$; notice that, since $\tilde\Lambda$ is locally maximal and $p\in\tilde\Lambda$, we have $q'\in \tilde\Lambda$ for $n$ large) has Markov value smaller than $t_0$, a contradiction. In order to do this, notice that if $r_k>2m^2 (k-k_n)$ for some $k$ which is cool and satisfies $k_n<k\le k_{n+1}$ or $s_k>2m^2(k_{n+1}-k)$ for some $k$ which is cool and satisfies and $k_n\le k<k_{n+1}$, then we get a contradiction by the previous Claim. Otherwise, except perhaps for the terms whose indices are multiple of $k_{n+1}-k_n$, any term equal
to $a_0$ in the periodic sequence with period $(a_{k_n}, a_{k_n+1},\dots,a_{k_{n+1}-1})$ has a neighbourhood in this
periodic sequence which coincides with the extended cell of a corresponding element of the
original sequence $(\dots, a_{-1}, a_0, a_1,\dots)$, and so the Markov value of the point corresponding
to this periodic sequence is smaller than $t_0$. For the terms whose indices are multiple of $k_{n+1}-k_n$ (which correspond to the point $q'$), if $\hat r$ is the positive integer such that $a_{k_{n+1}-j} = a_{-j}$ for $0 \le j < \hat r$ and $a_{k_{n+1}-\hat r}\ne a_{-\hat r}$, then $b_j=a_{k_{n+1}+j}$ for all $-\min\{\hat r, k_{n+1}-k_n+r_{k_n}\}<j\le 0$ and $b_j=a_{k_n+j}$ for $0\le j\le \tilde s_{k_n}$ (notice that $b_j=a_j$ for $0\le j<s_{k_n}$, but $b_{s_{k_n}}\ne a_{s_{k_n}}$), and so, by the definition of record, $d(\pi^s(a_{k_{n+1}-1}, a_{k_{n+1}-2},\dots), p)\le (1+o(1))|b/a|\lambda_1^3 d(\pi^u(a_{k_n+1}, a_{k_n+2},\dots), p)$ and thus $d(\pi^s(b_{-1}, b_{-2},\dots), p)\le (1+o(1))|b/a|\lambda_1 d(\pi^u(b_1, b_2,\dots), p)$, which implies $f(q)<f(p)=t_0$ (since $f(x,y) = ax + by + O(x^2 + y^2))$.

Now we concluded that $p$ and $\theta=(\dots, a_{-2}, a_{-1}, a_0, a_1, a_2,\dots)$ are periodic. Let $\alpha=(a_0, a_1,\dots, a_{s-1})$ be a minimal period of $\theta$. If $f\notin X$ then, for each positive integer $n$, there is a point $q_n\in \Lambda$ which does not belong to the orbit of $p$ such that $sup_{r\in \mathbb{Z}}f(\varphi^{r}(q_n))\le f(p)+1/n$. Given a kneading sequence $(\dots, c_{-2}, c_{-1}, c_0, c_1, c_2,\dots)$ of some point of $\Lambda$, we say that $k\in \mathbb{Z}$ is a {\it regular} position of it if there is an integer $j$ such that $c_{k+i}=a_{j+i\pmod s}$ for $1-s\le i\le 1$, and that $k$ is a {\it strange} position otherwise. Since $q_n$ does not belong to the orbit of $p$, there is $k_n\in \mathbb{Z}$ which is a strange position of the kneading sequence of $q_n$. Let $\tilde q_n:=\varphi^{k_n}(q_n)$. Then $0$ is a strange position of its kneading sequence. Take a subsequence of $(\tilde q_n)$ converging to a point $\tilde q\in \Lambda$. Then $0$ is a strange position of the kneading sequence of $\tilde q$ (and thus $\tilde q$ does belong to the orbit of $p$) and $sup_{r\in \mathbb{Z}}f(\varphi^{r}(\tilde q))\le f(p)$. This implies (since $f$ is injective in $\tilde\Lambda$ and $f(p)$ is the smallest element of $L(f,\Lambda)$) that $f(\varphi^{r}(\tilde q))<f(p), \forall r\in\mathbb Z$ and $\limsup_{r\to+\infty}f(\varphi^{r}(\tilde q))=\limsup_{r\to-\infty}f(\varphi^{r}(\tilde q))=f(p)$. 

Let $\tilde\theta=(\dots, b_{-2}, b_{-1}, b_0, b_1, b_2,\dots)$ be the kneading sequence of $\tilde q$, and let $m$ and $m_0$ be as in the Claim. We should have factors of $\tilde\theta$ with (arbitrarily large) positive indices equal to $\alpha^{2m}$, and, analogously, we should have factors of $\tilde\theta$ with negative indices equal to $\alpha^{2m}$. This implies that there is a factor of $\tilde\theta$ of the form $\alpha^{2m}\beta\alpha^{2m}$, where $\beta$ is a finite sequence which is not of the form $\alpha^r$ for any positive integer $r$, and which me may assume not to contain any factor of the form $\alpha^{2m}$ (otherwise we find a smaller factor with these properties). We claim that if $z\in\Lambda$ is a point whose kneading sequence is periodic with period $\alpha^{2m}\beta$ then $f(\varphi^r(z))<f(p), \forall r\in\mathbb Z$, and so, since $z$ is periodic, $sup_{r\in \mathbb{Z}}f(\varphi^{r}(z))<f(p)$, a contradiction. In order to do this, we use an argument somewhat analogous to the proof of the Claim: let $j$ such that $(b_j,b_{j+1},\dots,b_{j+M})=\alpha^{2m}\beta\alpha^{2m}$, where $M=4ms+b-1$, with $b=|\beta|$. Since $f(\varphi^{j+(2m-m_0)s-i}(\tilde q))<f(p)$ for $1\le i\le 2s$ and $f(\varphi^{j+(2m+m_0)s+b+i}(\tilde q))<f(p)$ for $1\le i\le 2s$, we have that $f(\pi^u(\alpha^j\beta\overline\alpha))<f(p)$ (so the suffix $\beta$ helps diminishing the value of $f$) and $f(\pi^s((\alpha^t)^j\beta^t\overline{\alpha^t}))<f(p)$ for every $j\ge 0$ (so the prefix $\beta$ helps diminishing the value of $f$). So, for the positions corresponding to $a_0$ in $\alpha^{2m}$, we use the fact that $\beta$ helps diminishing the value of $f$ both as a suffix and as a prefix in order to show that the corresponding values of $f$ are smaller than $f(p)$. For the other positions, including the positions inside $\beta$, we use the extended cell argument in order to show that the corresponding values of $f$ are also smaller than $f(p)$. This concludes the proof.
\qed
\bibliographystyle{alpha}	
\bibliography{Minima}

\begin{thebibliography}{AGCM18}

\bibitem[AGCM18]{CMM}
C.~Matheus{,} A.~G.~Cerqueira and C.~G. Moreira.
\newblock Continuity of {H}ausdorff dimension across generic dynamical
  {L}agrange and {M}arkov spectra.
\newblock {\em Journal of Modern Dynamics}, 12:151--174, 2018.

\bibitem[CF89]{CF}
T.~W. Cusick and M.~E. Flahive.
\newblock {\em The Markoff and Lagrange Spectra}.
\newblock Math surveys and Monographs. No 30, A.M.S., providence, RI, 1989.

\bibitem[Mar79]{Ma}
A.~Markov.
\newblock Sur les formes quadratiques binaires ind\'efinies.
\newblock {\em Math. Ann.}, 15:381--406, 1879.

\bibitem[MI16]{MR}
C.~G. Moreira and S.~A.~Roma{\~ n}a Ibarra.
\newblock On the {L}agrange and {M}arkov dynamical spectra.
\newblock {\em Ergodic Theory and Dynamical Systems}, 37(5):1570--1591, 2016.

\bibitem[MY10]{MY1}
Carlos~Gustavo Moreira and Jean-Christophe Yoccoz.
\newblock Tangencies homoclines stables pour des ensembles hyperboliques de
  grande dimension fractale.
\newblock {\em Annales Scientifiques de L' \' ecole Normale Sup\'erieure},
  43(4):1--68, 2010.

\bibitem[PT93]{PT}
J.~Palis and F.~Takens.
\newblock {\em Hyperbolicity \& sensitive chaotic dynamiscs at homoclinic
  bifurcations}.
\newblock Cambridge studies in abvanced mathematics, 35, 1993.

\bibitem[Shu86]{Shub}
Michael Shub.
\newblock {\em Global Stability of Dinamical Systems}.
\newblock Springer-Verlag, 1986.

\end{thebibliography}

\end{document}